\documentclass[preprint,11pt]{article}

\usepackage{amssymb}
\usepackage{latexsym,amsfonts,amssymb,amsmath,amsthm}
\usepackage{cite}
\usepackage{amsthm}
\usepackage{amsfonts}
\usepackage{xr}
\usepackage{amsmath}
\usepackage{color}
\usepackage{amssymb}
\usepackage{mdwlist}

\newtheorem{theorem}{Theorem}[section]
\newtheorem{definition}{Definition}[section]

\newtheorem{lemma}[theorem]{Lemma}
\newtheorem{rem}[theorem]{Remark}

\numberwithin{equation}{section}

  \def \G{\Gamma}

\makeatletter 
\@addtoreset{equation}{section}
\makeatother  

\newcommand\norm[1]{\lVert#1\rVert}

\newcommand\RR{\ensuremath{\mathbb{R}}}

\def\tint{\text{Int}}

\begin{document}
\title{Semiflows strongly focusing monotone with respect to high-rank cones. I: Generic dynamics}

\author{Lirui Feng \thanks{School of Mathematical Science, University of Science and Technology of China, Hefei, Anhui, 230026, Peoples Republic of China (ruilif@ustc.edu.cn). The author was supported by NSF of China No.12101583 and the Fundamental Research Funds of the Central Universities through grant WK0010000067.}}

\date{}
\maketitle

\begin{abstract}
We consider a smooth semiflow strongly focusing monotone with respect to a cone of rank $k$ on a Banach space.
We obtain its generic dynamics, that is, semiorbits with initial data from an open and dense subset of any bounded open set are either pseudo-ordered or convergent to an equilibrium. For the case $k=1$, it is the celebrated Hirsch's Generic Convergence Theorem. For the case $k=2$, we obtain the generic Poincar\'{e}-Bendixson Theorem.
\end{abstract}

\section{Introduction}

We investigate the generic dynamics of semiflows strongly focusing monotone with respect to a cone $C$ of rank $k$ on an infinite dimensional Banach space $X$. Roughly speaking, a cone of rank $k$ (abbr. $k$-cone) is a closed subset of $X$ containing a subspace of dimension $k$ but no subspace of higher dimension, which is independently introduced by Krasnosel'skij, Lifshits and Sobolev \cite{KLS} to do some functional studies, and by Fusco and Oliva \cite{FO1} to obtain a Perron theorem in the finite dimensional case. A convex cone $K$ gives rise to a 1-cone, $K\cup (-K)$. Therefore, the class of semiflows strongly monotone with respect to $k$-cones includes the classical monotone semiflows originating from the groundbreaking works of Hirsch (see \cite{Hir1,Hir2,Hir3,Hir33,Hir4,Hir5,Hir6}). Due to the lack of convexity for $k (k\geqslant 2)$-cones, it is a very challenging task to study the behaviors of this general class of systems strongly monotone with respect to $k$-cones. Despite some progress are achieved (see \cite{F-W-W1,F-W-W2,F-W-W3,F-W-W4}), their dynamics are far from being understood. To promote the development of systems monotone with respect to $k$-cones, we introduce the concept of semiflows strongly focusing monotone with respect to a $k$-cone $C$ and devote to describe their typical behaviors of almost all semiorbits in this paper.

A strongly focusing operator with respect to a $k$-cone $C$ originates from Krasnosel'skij et.al \cite{KLS} to prove a Krein-Rutman type theorem with respect to $k$-cones for a single operator, and also from Lian and Wang \cite{LW2} to investigate the relationship between Multiplicity Ergodic theorem and Krein-Rutman type theorem for random linear dynamical systems. Roughly speaking, the image of $C$ for a strongly focusing operator is a subset of $C$ such that its containing unit vectors are uniformly separated from the boundary of $C$  (see also Definition \ref{postive,strongly focusing}(ii)). We should point out that any strongly positive operator $R$ with respect to $C$ (i.e., $R C\subset \text{Int} C$) on a finite dimensional space is strongly focusing. Therefore, the smooth semiflows strongly focusing monotone with respect to $k$-cones on an infinite dimensional space is a kind of natural extension of the smooth flows with respect to $k$-cones on a finite dimensional space (refer to the flows in \cite{F-W-W2}). More precisely, this class of semiflows strongly monotone with respect to $C$ satisfies that for each compact invariant set $\Sigma$, one can find constants $\delta,T,\kappa>0$ such that there is a strongly focusing operator $T_{x,y}$ with separation index greater than $\kappa$ such that $T_{x,y}(y-x)=\Phi_T(y)-\Phi_T(x)$ for any $z\in \Sigma$ and $x,y\in B_{\delta}(z)$ (the ball centred at $z$ with radius $\delta$) (see also Definition \ref{strongly focusing monotone}(ii)). This class of semiflows would have significant potential applications to the study of dynamics of nonlinear evolution equations.

There are two types of semiorbits for the semiflow $\Phi_t$ monotone with respect to a $k$-cone $C$ (or a convex cone $K$): pseudo-ordered semiorbits and unordered semiorbits. A (positive) semiorbit $O^+(x)=\{\Phi_t(x):\,t\geq 0\}$ is pseudo-ordered if it contains a pair of different ordered points $\Phi_{\tau}(x)$ and $\Phi_s(x)$, i.e., $\Phi_{\tau}(x)-\Phi_{s}(x)\in C\setminus\{0\}$ (or $\Phi_{\tau}(x)-\Phi_{s}(x)\in K\setminus\{0\}$); otherwise, it is called unordered.

For the classical (in the sense of Hirsch's) monotone systems, the order reduced by a convex cone is a partial order relationship. Base on this fact, Monotone Convergence Criterion, a cornstone result in Hirsch's theory, can be established. It is to say that every pseudo-ordered precompact semiorbit converges to an equilibrium. The partial order also plays an important role in the further developments from Monotone Convergence Criterion, that includes Nonordering of Limit Sets and Limit Set Dichotomy. These results become building blocks (see \cite[Theorem 2.1, p.491]{Smi17}) for  establishing Hirsch's Generic Convergence Theorem.

Compared with classical monotone systems, the order reduced by $k$-cones is a symmetric relationship, that causes the structure of the omega-limit set $\omega(x)$ of a pseudo-ordered semiorbit $O^+(x)$ is more complicated and requires new techniques to analyze dynamics of the semiflows monotone with high-rank cones. Sanchez \cite{San09} firstly treated the problem on the structure of the omega-limit set $\omega(x)$ of a pseudo-ordered orbit for flows on $\mathbb{R}^n$ strongly monotone with respect to a $k(k\geq 2)$-cone $C$. He used the $C^1$-closed lemma to prove that any orbit in $\omega(x)$ of a pseudo-ordered orbit is ordered with respect to $C$; and further obtain a Poincar\'{e}-Bendixson theorem, that is, the omega-limit set $\omega(x)$ of a pseudo-ordered orbit containing no equilibrium is a closed orbit. For the total-ordering property of the entire set $\omega(x)$, he \cite[p.1984]{San09} posed it as an open problem. In our previous work \cite{F-W-W1}, we creatively utilized topological properties of continuous semiflows to study the total-ordering property for continuous semiflows strongly monotone with respect to a $k$-cone in a general Banach space and obtained the Order-Trichotomy (see \cite[Theorem B]{F-W-W1}) for the omega-limit set $\omega(x)$ of a pseudo-ordered orbit. More precisely, we proved that either (a) $\omega(x)$ is ordered; or (b) $\omega(x)$ belongs to the set of equilibria; or otherwise, (c) $\omega(x)$ possesses a certain ordered homoclinic property. In our previous work \cite[Theorem A and C]{F-W-W1}, we extended Sanchez's results to semiflows only continuous on an infinite dimensional space.

For semiflows strongly monotone with respect to higher rank cones, the symmetry of the order and the complexity of an omega-limit set $\Omega$ make it difficult to reappear the building blocks in Hirsch's theory. To treat the generic dynamics of flows $\Phi_t$ strongly monotone with respect to a $k$-cone $C$ on $\mathbb{R}^n$, we turned to analyze the local dynamics for each type of omega-limit set $\Omega$ in our previous works \cite{F-W-W2}, where the types are classified by our approach of smooth ergodic arguments. More precisely, the linear skew-product flow $(\Phi_t, D\Phi_t)$ admits $k$-exponential separation along $\Omega$ associated with $C$ provided strongly positivity of $D_x\Phi_t$ for any $x\in \Omega$ and $t>0$. Roughly speaking, this property describes that there exist $k$-dimensional invariant subbundle $\Omega\times (E_x)$ and $k$-codimensional invariant subbundle $\Omega\times (F_x)$ with respect to $(\Phi_t, D\Phi_t)$ such that $\Omega\times \mathbb{R}^n=\Omega\times (E_x)\oplus \Omega\times (F_x)$; and more, the action of $(\Phi_t, D\Phi_t)$ on $\Omega\times (E_x)$ dominates the one on $\Omega\times (F_x)$ as $t\rightarrow \infty$ (see Definition \ref{D:ES-separation} or its versions for random dynamics in \cite{LW1,LW2}). The related crucial tool is the $k$-Lyapunov exponent $\lambda_{kx}$ of $x\in \Omega$ (defined as $\lambda_{kx}=\limsup\limits_{t\rightarrow+\infty}\frac{\log m(D_x\Phi_t|_{E_x})}{t}$, see also Definition \ref{Lyapunov expnents} and (\ref{infimum norm})), which describes the action's growth rate of $(\Phi_t, D\Phi_t)$ on the $k$-dimensional subbundle $\Omega\times (E_x)$. Multiplicity Ergodic theorem ensures \cite{Jo-Pa-Se,LL,Mane,Ose,Vi} that $\lambda_{kx}$ is actually the limit for ??most?? points $x\in \Omega$; such points for which $\lambda_{kx}$ is the limit are said to be regular and other points are said to be irregular. According to the sign of the $k$-Lyapunov exponents of the regular/irregular points on any given $\Omega$, three are three types of a given omega-limit set $\Omega$: (i) $\lambda_{kx}>0$ for any point $x$ in $\Omega$; (ii) $\lambda_{kx}>0$ for any regular point $x\in \Omega$ and $\lambda_{kz}\leq 0$ for some irregular point $z\in \Omega$; (iii) $\lambda_{kx}\leq 0$ for some regular point $x\in \Omega$. By discussing local behaviors for each type of omega-limit sets, we obtain the finite dimensional version of Generic dynamics theorem (see \cite[Theorem A]{F-W-W2}); and further combinate it with the Poincar\'{e}-Bendixson theorem (see Lemma \ref{P-B thm} and also \cite[Theorem C]{F-W-W1}) of the omega-limit set of a pseudo-ordered semiorbit to get the generic Poincar\'{e}-Bendixson theorem (see \cite[Theorem B]{F-W-W2}) on $\mathbb{R}^n$.

The purpose of this paper is to investigate the infinite dimensional version of generic dynamics of semiflow $\Phi_t$ strongly focusing monotone with respect to a $k$-cone $C$ on a Banach space. We prove that

$\bullet$ For generic (open and dense) semiorbits are pseudo-ordered or converge to an equilibrium.

$\bullet$ Whenever $k=2$, for generic points, the omega-limit set containing no equilibrium is a periodic orbit.

\noindent By the strongly positivity of $D_x\Phi_t$ in Definition \ref{strongly focusing monotone}(i), the linear skew-product semiflow $(\Phi_t,D\Phi_t)$ admits $k$-exponential separation $\Sigma\times X=\Sigma\times(E_x)\oplus \Sigma\times(F_x)$ along each invariant set $\Sigma$ associated with $C$. However, the unit ball in an infinite dimensional Banach space is lack of compactness; and hence, unit vectors in the $k$-codimensional invariant subbundle $\Sigma\times (F_x)$ with respect to $(\Phi_t,D\Phi_t)$ are not uniformly far away from the boundary $\partial C$ of the $k$-cone $C$. The method in \cite{F-W-W2} is not effective to estimate the proportion of the projection onto the invariant bundles $\Sigma\times(E_x)$ and $\Sigma\times(F_x)$ for a nonzero vector $v\in C$ in an infinite dimensional space. We turn to estimate the proportion of the projection onto the invariant bundles $\Sigma\times(E_x)$ and $\Sigma\times(F_x)$ for the difference $\Phi_t(x)-\Phi_t(y)$ of a pair of ordered points $\Phi_t(x)$ and $\Phi_t(y)$ by utilizing the strongly focusing condition in Definition \ref{strongly focusing monotone}(ii). By this novel approach, we analyze the local dynamical features for each type of omega-limit sets and furthermore deduce the infinite dimensional version of generic dynamics. For case $k=2$, the generic Poincar\'{e}-Bendixson theorem are also obtained.

The paper is organized as follows. In section 2, we give some notations and summarize the preliminary results. In section 3, we present the main results on the infinite dimensional version of Generic dynamics and generic Poincar\'{e}-Bendixson theorem for semiflows strongly focusing monotone with respect to a $k$-cone. In section 4, we discuss the local behaviors for each type of omega-limit sets. In section 5, we prove our main results.

\section{Notations and Preliminary Results}

In this section, we give some preliminary knowledge to be used in later sections. We start with basic nations and definitions on semiflows strongly focusing monotone with respect to higher rank cones. We then introduce the $k$-exponential separation and $k$-Lyapunov exponents with some crucial properties of them.
\subsection{Semiflows strongly monotone with respect to higher rank cones}

Let $(X,\norm{\cdot})$ be a Banach space equipped with a norm
$\norm{\cdot}$. A {\it semiflow} on $X$ is a continuous map $\Phi:\mathbb{R}^+\times X\to X$ with $\Phi_0={\rm Id}$ and
$\Phi_t\circ\Phi_s=\Phi_{t+s}$ for $t,s\ge 0$. Here, $\mathbb{R}^+=[0,+\infty)$, $\Phi_t(\cdot)=\Phi(t,\cdot)$ for $t\ge 0$, and ${\rm Id}$
is the identity map on $X$. A semiflow $\Phi_t$ on $X$ is called {\it $C^{1,\alpha}$-smooth} if $\Phi|_{\mathbb{R}^+\times X}$ is a
$C^{1,\alpha}$-map (a $C^1$-map with a locally $\alpha$-H\"{o}lder derivative) with $\alpha\in (0,1]$. The derivatives of $\Phi_t$ with respect to
$x$, at $(t,x)$, is denoted by $D_x\Phi_t(x)$.

Let $x\in X$,  the {\it positive semiorbit of $x$} is denoted by $O^+(x)=\{\Phi_t(x):t\ge 0\}$. A {\it negative semiorbit}
(resp. {\it full-orbit}) of $x$ is a continuous function $\psi:\mathbb{R}^-=\{t\in \mathbb{R}|t\le 0\}\to X$ (resp. $\psi:\mathbb{R}\to X$) such
that $\psi(0)=x$ and, for any $s\le 0$ (resp. $s\in \mathbb{R}$), $\Phi_t(\psi(s))=\psi(t+s)$ holds for $0\le t\le -s$ (resp. $0\le t$). Clearly,
if $\psi$ is a negative semiorbit of $x$, then $\psi$ can be extended to a full orbit
\begin{equation}\label{E:full-orbit}
\tilde{\psi}(t)=\left\{\begin{split} & \psi(t), \,\,\,\,\,t\le 0,\\ &\Phi_t(x),\,\, t\ge 0.\end{split}\right.
\end{equation}
On the other hand, any full orbit of $x$ when restricted to $\mathbb{R}^-$ is a negative semiorbit of $x$. Since $\Phi_t$ is just a semiflow,
a negative semi-orbit of $x$ may not exist, and it is not necessary to be unique even if one exists.

An {\it equilibrium} (also called {\it a trivial orbit}) is a point $x$ for which $O^+(x)=\{x\}$. Let $E$ be the set of all equilibria of $\Phi_t$.
A nontrivial semiorbit $O^+(x)$ is said to be a {\it$T$-periodic orbit} for some $T>0$ if $\Phi_T(x)=x$.

A subset $\Sigma\subset X$ is called {\it positively invariant} if $\Phi_{t}(\Sigma)\subset \Sigma$ for any $t\in \mathbb{R}^+$, and is called {\it invariant}
if $\Phi_{t}(\Sigma)=\Sigma$ for any $t\in \mathbb{R}$. Clearly, for any $x\in \Sigma$, there exists a negative semi-orbit of $x$, provided that $\Sigma$ is invariant. Let $\Sigma\subset X$ be an invariant set. $\Phi_t$ is said to {\it admit a flow extension on $\Sigma$}, if there is a flow $\tilde{\Phi}_t$ such that $\tilde{\Phi}_t(x)=\Phi_t(x)$ for any $x\in \Sigma$ and $t\ge 0.$

The {\it$\omega$-limit set} $\omega(x)$ of $x\in X$ is defined by $\omega(x)=\cap_{s\ge 0}\overline{\cup_{t\ge s}\Phi_t(x)}$. If $O^+(x)$ is
precompact, then $\omega(x)$ is nonempty, compact, connected and invariant. Given a subset $D\subset X$, {\it the positive semiorbit $O^+(D)$ of $D$} is defined as $O^+(D)=\bigcup\limits_{x\in D}O^+(x)$. A subset $D$ is called {\it $\omega$-compact} if $O^+(x)$ is precompact for each
$x\in D$ and $\bigcup\limits_{x\in D}\omega(x)$ is precompact. Clearly, $D$ is $\omega$-compact provided by the compactness of $\overline{O^+(D)}$.

A closed set $C\subset X$ is called a cone of rank-$k$ ({\it abbr. $k$-cone}) if\\
\indent {\rm i)} For any $v\in C$ and $l\in \RR,$ $lv\in C$; \\
\indent {\rm ii)}  $\max\{\dim W:C\supset W \text{ linear subspace}\}=k.$\\
\noindent Moreover, the integer $k(\ge 1)$ is called the rank of $C$. A $k$-cone $C\subset X$ is said to be {\it solid} if its interior
$\tint C\ne \emptyset$; and $C$ is called {\it $k$-solid} if there is a $k$-dimensional linear subspace $W$ such that
$W\setminus \{0\}\subset \tint C$. Given a $k$-cone $C\subset X$, we say that $C$ is {\it complemented} if there exists a $k$-codimensional space
$H^{c}\subset X$ such that $H^{c}\cap C=\{0\}$. For two points $x,y\in X$, we call that {\it $x$ and $y$ are ordered, denoted by $x\thicksim y$}, if $x-y\in C$. Otherwise, $x,y$ are called to be {\it unordered}. The pair $x,y\in X$ are said to be {\it strongly ordered},
denoted by $x\thickapprox y$, if $x-y\in \tint C$.

Let $d(x,y)=\norm{x-y}$ for any $x,y\in X$ and $d(x,B)=\inf\limits_{y\in B}d(x,y)$ for any $x\in X, B\subset X$.\\

Throughout this paper, we assume $C$ is a $k$-solid complemented cone and $\Phi_t$ with compact $x$-derivartive $D_x\Phi_t$ for $(x,t)\in X\times \mathbb{R}^+$ admits a flow extension on each omega-limit set $\omega(x)$.
\begin{definition}\label{postive,strongly focusing}

{\rm(i) A linear operator $R\in L(X)$ is called {\it strongly positive with respect to $C$}, if
$R\,\big(C\setminus\{0\}\big)\subset \text{Int} C$.

(ii) A linear operator $R$ is called {\it strongly focusing with respect to $C$} if there is a $\kappa>0$ such that $$\underline{\text{dist}}(RC, X\setminus C)=\kappa,$$ where $\underline{\text{dist}}(L_1,L_2)$ is the separation index between set $L_1$ and $L_2$ defined by $$ \underline{\text{dist}}(L_1,L_2)=\inf\limits_{v\in L_1, \norm{v}=1}\{\inf\limits_{u\in L_2}\norm{v-u}\}.$$ Here, $\kappa$ is also called the separation index of $R$.}
\end{definition}
\begin{rem} {\rm (i) A strongly focusing operator is automatically a strongly positive operator.

(ii) Let $R$ be a strongly positive operator w.r.t. $C$ on $\mathbb{R}^n$. Then, $R$ is also strongly focusing w.r.t. $C$.}
\end{rem}

A semiflow $\Phi_t$ on $X$ is called {\it monotone with respect to $C$} if
$$\Phi_t(x)\thicksim\Phi_t(y),\, \text{ whenever } x\thicksim y \text{ and } t\ge 0;$$ and $\Phi_t$ is called {\it strongly monotone with respect
to $C$} if $\Phi_t$ is monotone with respect to $C$ and $$\Phi_t(x)\approx \Phi_t(y),\, \text{ whenever } x\ne y, x\thicksim y \text{ and } t>0.$$ A nontrivial positive semiorbit $O^+(x)$ is called {\it pseudo-ordered} (also called {\it of Type-I}), if there exist two distinct points
$\Phi_{t_1}(x),\Phi_{t_2}(x)$ in $O^+(x)$ such that $\Phi_{t_1}(x)\thicksim\Phi_{t_2}(x)$. Otherwise, $O^+(x)$ is called {\it unordered}
(also called {\it of Type-II}). Hereafter, we let $$Q=\{x\in X: O^+(x) \text{ is pseudo-ordered}\}.$$

\begin{definition}\label{strongly focusing monotone} {\rm A semiflow $\Phi_t$ is called {\it strongly focusing monotone with respect to $C$}, if it satisfies:

(i) It is $C^1$-smooth and strongly monotone with respect to $C$ such that the $x$-derivative $D_x\Phi_t$ of $\Phi_t \,(t>0)$ is strongly positive with respect to $C$ for any $x\in X$;

(ii) For each compact invariant set $\Sigma$ of $\Phi_t$, one can find constants $\delta,T,\kappa>0$ such that there is a strongly focusing operator $T_{x,y}$ with separation index greater than $\kappa$ such that $T_{x,y}(y-x)=\Phi_T(y)-\Phi_T(x)$ for any $z\in \Sigma$ and $x, y\in B_{\delta}(z)$, where $B_{\delta}(z)$ is the ball centred at $z$ with radius $\delta$.}
\end{definition}

\begin{rem} {\rm Let $\Phi_t$ be a $C^1$-smooth flow strongly monotone w.r.t. $C$ on $\mathbb{R}^n$, whose $x$-derivative $D_x\Phi_t$ is strongly positive w.r.t. $C$ for any $x\in\mathbb{R}^n$ and $t>0$. Then, $\Phi_t$ is strongly focusing monotone w.r.t. $C$.}

\end{rem}

\begin{rem} {\rm The strongly focusing condition in Definition \ref{strongly focusing monotone}(ii) can be relexed and only required for each omega limit set $\omega(x)$ in the proof of the results in this paper.}
\end{rem}

\begin{rem}{\rm Let $\tilde{\Sigma}=\text{Co}\{B_{\delta}(\Sigma)\times\text{Co}\{B_{\delta}(\Sigma)\}$. Here, $B_{\delta}(\Sigma)=\{v\in X: d(v,\Sigma)\leq\delta\}$ and $\text{Co}\{B_{\delta}(\Sigma)\}$ is the convex hull of $B_{\delta}(\Sigma)$. Let $T_{x,y}=\int_0^1D_{x+s(y-x)}\Phi_Tds$ for any $(x,y)\in\tilde{\Sigma}$. Then, one has $T_{x,y}(y-x)=\Phi_T(y)-\Phi_T(x)$. Let $\kappa>0$. Compared with Definition \ref{strongly focusing monotone}(ii), the following condition has more restriction. $$(*)\,\,\,\{T_{x,y}\}_{\tilde{\Sigma}}\,\,\text{is a family of strongly focusing operators with separation index greater than}\,\,\kappa.$$}
\end{rem}

Now, we give several useful results on semiflows strongly monotone with respect to $C$.

\begin{lemma}\label{co-limit}
Assume that $\Phi_t$ is strongly monotone with respect to $C$. If $x\thicksim y$ and there is a sequence $t_n\to \infty$ such that
$\Phi_{t_n}(x)\rightarrow z$ and $\Phi_{t_n}(y)\rightarrow z$, then $z\in Q$ or $z$ is an equilibrium.
\end{lemma}

\begin{proof}See \cite[Lemma 4.3]{F-W-W1}.

\end{proof}

\begin{lemma}\label{P-B thm} Assume that $\Phi_t$ is strongly monotone with respect to a $k$-cone $C$ with $k=2$ and $O^+(x)$ be a nontrivial pseudo-ordered precompact semiorbit. If $\omega(x)\cap E=\emptyset$, $\omega(x)$ is a periodic orbit.
\end{lemma}
\begin{proof} See \cite[Theorem C]{F-W-W1}.
\end{proof}

\subsection{$k$-Exponential Separation and $k$-Lyapunov Exponents}

Let $G(k,X)$ be {\it the Grassmanian of $k$-dimensional linear subspaces of $X$}, which consists of all $k$-dimensional linear subspace in $X$.
$G(k, X)$ is a completed metric space by endowing {\it the gap metric} (see, for example, \cite{Kato,LL}). More precisely, for any nontrivial
closed subspaces $L_1,L_2\subset X$, define that \begin{equation*}\label{E:GapDistance}
 d(L_1,L_2)=\max\left\{\sup_{v\in L_1\cap S}\inf_{u\in L_2\cap S}\norm{v-u}, \sup_{v\in L_2\cap S}\inf_{u\in L_1\cap S}\norm{v-u}\right\},
 \end{equation*} where $S=\{v\in X:\norm{v}=1\}$ is the unit ball.
For a solid $k$-cone $C\subset X$, we denote by $\G_k(C)$ the set of $k$-dimensional subspaces inside $C$, that is,
$$\G_k(C)=\{L\in G(k,X):\,L\subset C\}.$$
Let $\Sigma\subset X$ be a compact invariant subset for $\Phi_t$ and we consider the linear skew-product semiflow $(\Phi_t,\,D\Phi_t)$ on $\Sigma\times X$,
which is defined as $(\Phi_t, D\Phi_t)(x,v)=(\Phi_t(x),D_x\Phi_t v)$ for any $(x,v)\in \Sigma\times X$ and $t>0$. Here, $D_x\Phi_t$ is the Fr\'{e}chet derivative of $\Phi_t$ at $x\in \Sigma$. Let $\{E_x\}_{x\in \Sigma}$ be a family of $k$-dimensional subspaces of $X$. We call $\Sigma\times (E_x)$ {\it a $k$-dimensional continuous vector bundle on $X$} if the map $\Sigma\mapsto G(k,\,X): x\mapsto E_x$ is continuous. Let $\{F_x\}_{x\in \Sigma}$ be a family of $k$-codimensional closed vector subspaces of $X$. We call $\Sigma\times (F_x)$ {\it a $k$-codimensional continuous vector bundle on $X$} if there is a $k$-dimensional continuous vector bundle $\Sigma\times (L_x)\subset \Sigma\times X^*$ such that the kernel ${\rm Ker}(L_x)=F_x$ for each $x\in \Sigma$.
Here, $X^*$ is the dual space of $X$.

Let $\Sigma\times (E_x)$ be a $k$-dimensional continuous vector bundle on $X$, and let $\Sigma\times (F_x)$ be a $k$-codimensional continuous vector bundle
on $X$ such that $X=E_x\oplus F_x$ for all  $x\in \Sigma$. We define the {\it family of projections associated with the decomposition}
$X=E_x\oplus F_x$ as $\{P_x\}_{x\in \Sigma}$ where $P_x$ is the linear projection of $X$ onto $E_x$ along $F_x$, for each $x\in \Sigma$.
Write $Q_x=I-P_x$ for each $x\in \Sigma$. Clearly, $Q_x$ is the linear projection of $X$ onto $F_x$ along $E_x$. Moreover,
both $P_x$ and $Q_x$ are continuous with respect to $x\in \Sigma$. We say that the decomposition $X=E_x\oplus F_x$ is
{\it invariant with respect to $(\Phi_t,\,D\Phi_t)$} if $D_x\Phi_tE_x=E_{\Phi_t(x)}$, $D_x\Phi_tF_{x}\subset F_{\Phi_t(x)}$ for each $x\in \Sigma$
and $t\ge 0$.

\begin{definition}\label{D:ES-separation}
{\rm Let $\Sigma\subset X$ be a compact invariant subset for $\Phi_t$. The linear skew-product semiflow $(\Phi_t,\,D\Phi_t)$ admits a
{\bf $k$-exponential separation along $\Sigma$} (for short, $k$-exponential separation), if there are $k$-dimensional continuous bundle
$\Sigma\times (E_x)$ and $k$-codimensional continuous bundle $\Sigma\times (F_x)$ such that

{\rm (i)} $X=E_x\oplus F_x$, for any $x\in \Sigma$;

{\rm (ii)} $D_x\Phi_tE_x=E_{\Phi_t(x)}$, $D_x\Phi_tF_{x}\subset F_{\Phi_t(x)}$ for any $x\in \Sigma$ and $t>0$;

{\rm (iii)}  there are constants $M>0$ and $0<\gamma<1$ such that
$$\norm{D_x\Phi_tw}\leq M\gamma^{t}\norm{D_x\Phi_tv}$$ for all $x\in \Sigma$, $w\in F_x\cap S$, $v\in E_{x}\cap S$
and $t\ge 0$, where $S=\{v\in X:\norm{v}=1\}$.

Let $C\subset X$ be a solid $k$-cone. If, in addition,

{\rm (iv)} $E_x\subset {\rm Int}C\cup \{0\}$ and $F_x\cap C=\{0\}$ for any $x\in \Sigma$,

\noindent then $(\Phi_t,\,D\Phi_t)$ is said to admit a {\bf $k$-exponential separation along $\Sigma$ associated with $C$}.}
\end{definition}

Since $E_x$ is $k$ dimensional for any $x\in \Sigma$, one can define {\it the infimum norm $m(D_x\Phi_t|_{_{E_x}})$
of $D_x\Phi_t$ restricted to $E_x$} for each $x\in \Sigma$ and $t\geq 0$ as follows:
\begin{equation}\label{infimum norm}m(D_x\Phi_t|_{_{E_x}})=\inf\limits_{v\in E_x\cap S}\norm{D_x\Phi_tv},\end{equation} where $S=\{v\in X: \norm{v}=1\}$.

\begin{definition}\label{Lyapunov expnents}{\rm For each $x\in \Sigma$, {\it the $k$-Lyapunov exponent} is defined as
\begin{equation}
\lambda_{kx}=\limsup_{t\to +\infty}\dfrac{\log m(D_x\Phi_t|_{_{E_x}})}{t}.
\end{equation}}
\end{definition} A point $x\in \Sigma$ is called {\it a regular point} if
$\lambda_{kx}=\lim\limits_{t\to +\infty}\dfrac{\log m(D_x\Phi_t|_{_{E_x}})}{t}$.
\vskip 3mm

\begin{lemma}\label{P:Cones-imply-ES} Assume that $\Sigma\subset X$ be a compact invariant subset with respect to $\Phi_t$ and $\Phi_t$ is $C^1$-smooth such that $D_x\Phi_t (C\setminus\{0\})\subset \text{Int}\,C$ for any $x\in \Sigma$ and $t>0$. Then, $(\Phi_t,D\Phi_t)$ admits a $k$-exponential separation along $\Sigma$ associated with $C$.
\end{lemma}
\begin{proof}See Tere$\check{s}\check{c}$\'{a}k \cite[Corollary 2.2]{Te1}. One may also refer to Tere$\check{s}\check{c}$\'{a}k \cite[Theorem 4.1]{Te1}.
\end{proof}
Now, we give some crucial lemmas for $(\Phi_t,D\Phi_t)$ admitting a $k$-exponential separation $X=E_y\oplus F_y$ along a compact invariant set $\Sigma$ associated with $C$, which satisfies (i)-(iv) in Definition \ref{D:ES-separation}.

\begin{lemma}\label{vector-in-C-farawayfrom-V1-delta} There exists a constant $\delta^{\prime}>0$ such that
$$\{v\in X: d(v,\,E_x\cap S)\leq\delta^{\prime}\}\subset\text{Int}\,C \,\,\,\text{for any}\,\,x\in \Sigma.$$
\end{lemma}
\begin{proof} See \cite[Lemma 3.3]{F-W-W2}. We here point out that all arguments in \cite[Lemma 3.3]{F-W-W2} still remain valid for $C^{1}$-semiflow $\Phi_t$ on a Banach space.
\end{proof}

\begin{lemma}\label{L:ES-Vs-Cone}
{\rm (i)} The projections $P_x$ and $Q_x$ are bounded uniformly for $x\in \Sigma$.

{\rm (ii)}  There exists a constant $C_1>0$ such that, if $v\in X\setminus\{0\}$ satisfies $\norm{P_x(v)}\geq C_1 \norm{Q_x(v)}$ for some $x\in \Sigma$, then $v\in \text{Int}\,C$.
\end{lemma}
\begin{proof} See \cite[Lemma 3.5(i) and (ii)]{F-W-W2}. The arguments in \cite[Lemma 3.5(i) and (ii)]{F-W-W2} are also effective for semiflow $\Phi_t$ on a Banach space.
\end{proof}

\begin{lemma}\label{E-S-and-Lya-exponent} Let $x\in \Sigma$. Then

{\rm (i)} If $w\in F_x\setminus\{0\}$, then $\lambda(x,w)\leq \lambda_{kx}+\log(\gamma)$, where $\lambda(x,w)=\limsup\limits_{t\to \infty}
\frac{\log\norm{D_x\Phi_t w}}{t}.$

{\rm (ii)} Let $x$ be a regular point. If $\lambda_{kx}\leq 0$, then there exists a number $\beta\in (\gamma,1)$ such that for any $\epsilon>0$,
there is a constant $C_{\epsilon}>0$ such that $$\norm{D_{\Phi_{t_1} (x)}\Phi_{t_2}w}\leq C_{\epsilon}e^{\epsilon t_1}\beta^{t_2}\norm{w}$$
for any $w\in F_{\Phi_{t_1}(x)}\setminus\{0\}$ and $t_1,t_2>0$.
\end{lemma}

\begin{proof} The results are directly implied by repeating all arguments in \cite[Lemma 3.6]{F-W-W2}
\end{proof}
\begin{rem}\label{the effection of Lemmas} Lemma \ref{vector-in-C-farawayfrom-V1-delta}-\ref{E-S-and-Lya-exponent} are the infinite demensional version of \cite[Lemma 3.3,3.5(i)-(ii),3.6]{F-W-W2}. Lemma \ref{co-limit} and them are crucial tools for the arguments of Theorem \ref{regular-lambdak>0} and Lemma \ref{oribit-lambdak<0}.
\end{rem}

\section{Main results}
\indent Let $C_E=\{x\in X:\, \text{the positive semiorbit}\, O^+(x) \text{ converges to an equilibrium}\}.$

{\bf Theorem A. (Generic dynamics thoerem)} Assume that $\Phi_t$ is a $C^{1,\alpha}$-smooth semiflow strongly focusing monotone with respect to a $k$-cone $C$. Let $\mathcal{D}\subset X$ be an open bounded set such that $\mathcal{O^+}(\mathcal{D})$ is precompact. Then ${\rm Int}(Q\cup C_E)$ {\rm ({\it interior in $X$})} is dense in $\mathcal{D}$.

\begin{rem}
{\rm Theorem A states that, for smooth semiflow $\Phi_t$ strongly focusing monotone with respect to $k$-cone $C$, generic (open and dense) positive semiorbits are either pseudo-ordered or convergent to equilibria.
If the rank $k=1$, Theorem A automatically implies Hirsch's Generic Convergence Theorem due to the Monotone Convergence Criterion.}
\end{rem}
\vskip 3mm

{\bf Theorem B. (Generic Poincar\'{e}-Bendixson theorem)} Assume that $\Phi_t$ is a $C^{1,\alpha}$-smooth semiflow strongly focusing monotone with respect to a $k$-cone $C$. Let $k=2$ and $\mathcal{D}\subset X$ be an open bounded set such that $\mathcal{O}^+(\mathcal{D})$ is precomact. Then, for generic (open and dense) points $x\in \mathcal{D}$, the omega-limit set $\omega(x)$ containing no equilibria is a periodic orbit.

\section{Local behaviors of omega-limit sets}
Due to Lemma \ref{P:Cones-imply-ES}, we hereafter always assume that the linear skew-product semiflow $(\Phi_t,D\Phi_t)$ admits a $k$-exponential separation along each compact invariant set $\Sigma$ such that $X=E_x\oplus F_x$. $\Sigma \times (E_x)$ and $\Sigma \times (F_x)$ are the corresponding $k$-dimensional and $k$-codimensional invariant subbundles. In this paper, we attempt to extend the works on generic dynamics from classical monotone systems w.r.t. convex cones (see \cite{Hir2}) and flows strongly monotone w.r.t. $k$-cones (see \cite{F-W-W2}) to the semiflows strongly focusing monotone w.r.t. $k$-cones on an infinite dimensional Banach space.

We define the set of {\it regular points} on $\omega(x)$ as:
\begin{equation}\label{E:rg-point}
\omega_0(x)=\{z\in\omega(x):\,z\,\, \text{is a regular point}\}.\end{equation}
\noindent Due to the Multiplicative Ergodic Theorem (cf. \cite[Theorem A]{Mane}), $\omega_0(x)$ is non-empty. Moreover, it is easy to see that any equilibrium in $\omega(x)$ is regular and hence, is contained in $\omega_0(x)$. By utilizing the $k$-Lyapunov exponents on $\omega(x)$, we classify the omega-limit sets into three types and obtain the related local behaviors.

Firstly, we prove that if $\lambda_{kz}>0$ for any $z\in \omega(x)$, then $x$ is highly unstable (see Lemma \ref{Distance-nonlinear-system}), and meanwhile, it belongs to the closure $\overline{Q}$ (see Theorem \ref{omega-lambdak>0}). We secondly show that if $\lambda_{k\tilde{z}}>0$ for any regular point $\tilde{z}\in \omega_0(x)$ but there is an irregular point $z$ with $\lambda_{kz}\leq 0$, then $x\in \overline{Q}$ (see Theorem \ref{regular-lambdak>0}). We finally show that if $\omega(x)$ contains a regular point $z$ such that $\lambda_{kz}\le 0$, then either $x\in Q$ or $\omega(x)$ is a singleton (see Theorem \ref{omega-lambdak<=0}).

\vskip 2mm
We start with discussion on the case that $\lambda_{kz}>0$ for any point $z\in \omega(x)$. Before going further, we give two technical lemmas.

\begin{lemma}\label{expand-in-Ez} If $\lambda_{kz}>0$ for any $z\in \omega(x)$, then for any constant $\kappa>0$, there is a locally constant
(hence bounded) function $\nu_{\kappa}(z)$ on $\omega(x)$ (depending on $\kappa$) such that
\begin{equation}\label{eq:nu-kappa}\begin{aligned}&\frac{\norm{D_z\Phi_{\nu_{\kappa}(z)}w_F}}{\norm{D_z\Phi_{\nu_{\kappa}(z)}w_E}}
<\frac{\kappa}{2(1+\kappa)},\\&\norm{D_z\Phi_{\nu_{\kappa}(z)}w_E}>\frac{4}{\kappa}\end{aligned}\end{equation}
for any $z\in \omega(x)$ and $w_E\in E_z\cap S$ and $w_F\in F_z\cap S$, where $S=\{v\in X: \norm{v}=1\}$.
\end{lemma}

\begin{proof} By the definition of $\lambda_{kz}$, for each $z\in \omega(x)$, there is a sequence $t_n\rightarrow+\infty$ such that
$$\norm{D_z\Phi_{t_n}w_E}>e^{\frac{\lambda_{kz}}{2}t_n}$$ for any $w_E\in E_z\cap S$. Furthermore, the definition of $k$-exponential separation
along $\Sigma$ indicates that there exist $M>0$ and $\gamma\in(0,1)$ such that $$\frac{\norm{D_z\Phi_{t}w_F}}{\norm{D_z\Phi_{t}w_E}}< M\gamma^t$$ for
any $t>0$ and $w_E\in E_z\cap S,\,w_F\in F_z\cap S$. Since $\lambda_{kz}>0$, one can find a $N_{\kappa}(z)>0$ such that
$$\frac{\norm{D_z\Phi_{t_n}w_F}}{\norm{D_z\Phi_{t_n}w_E}}<\frac{\kappa}{2(1+\kappa)}$$ and $$\norm{D_z\Phi_{t_n}w_E}>\frac{4}{\kappa}$$ for any
$t_n>N_{\kappa}(z)$ and $w_E\in E_z\cap S$.

Therefore, for each $z\in \omega(x)$, one can associate with a number $\nu_{\kappa}(z)\geq N_{\kappa}(z)$ such that (\ref{eq:nu-kappa}) holds for
any $w_E\in E_z\cap S$ and $z\in\omega(x)$. Moreover, together with the compactness of $\omega(x)$ and the smoothness of $\Phi_t$, one can
further take such $\nu_{\kappa}(z)$ as a locally constant (hence bounded) function. We complete the proof.
\end{proof}

\begin{lemma}\label{Distance-nonlinear-system} Assume that $\lambda_{kz}>0$ for any $z\in \omega(x)$. There exists a constant
$\delta^{\prime\prime}>0$ such that $$\limsup_{t\to +\infty}\norm{\Phi_t(y)-\Phi_t(x)}\ge \delta^{\prime\prime},$$ whenever $y$
satisfies $y\neq x$ and $y\thicksim x$.
\end{lemma}

\begin{proof} Since $\Phi_t$ is strongly focusing monotone w.r.t. $C$, one can find constants $\delta,T,\kappa>0$ such that there exists a strongly focusing operator $T_{x,y}$ with separation index greater than $\kappa$ such that $T_{x,y}(y-x)=\Phi_T(y)-\Phi_T(x)$ for any $z\in \omega(x)$ and $x,y\in B_{\delta}(z)$, where $B_{\delta}(z)$ is the ball centred at $z$ with radius $\delta$.

For any given $y\in X$ such that $y\neq x$ and $y\thicksim x$, let $\tilde{y}_{t}=\Phi_{t}(y)$ and
$\tilde{x}_{t}=\Phi_{t}(x)$ for $t>0$. Since $\omega(x)$ attracts $x$, one can take a curve $\{z_t\}_{t>0}\subset \omega(x)$ such that
\begin{equation}\lim\limits_{t\rightarrow\infty}\norm{\tilde{x}_t-z_t}=0.\end{equation} Hence, there exists a $\tilde{T}_y>T$ such that $\norm{\tilde{x}_{t-T}-z_{t-T}}<\frac{\delta}{2}$ for any $t>\tilde{T}_y$. Moreover, if $\tilde{y}_{t-T}\in B_{\frac{\delta}{2}}(z_{t-T})$ for some $t>\tilde{T}_y$, then there is a strongly focusing operator $T_{\tilde{x}_{t-T},\tilde{y}_{t-T}}$ with separation index greater than $\kappa$ such that $$\tilde{y}_t-\tilde{x}_t=T_{\tilde{x}_{t-T},\tilde{y}_{t-T}}(\tilde{y}_{t-T}-\tilde{x}_{t-T}).$$ By Lemma \ref{P:Cones-imply-ES}, $(\Phi_t, D\Phi_t)$ admits a $k$-exponential separation along $\omega(x)$ associated with $C$. Denoted by $\omega(x)\times (E_z)$ ($\omega(x)\times (F_z)$) the corresponding $k$-dimensional invariant subbundle ($k$-codimensional subbundle), respectively. Recall that $P_z$ is the projection onto $E_z$ along $F_z$ and $Q_z=I-P_z$. Clearly, $d(\frac{\tilde{y}_t-\tilde{x}_t}{\norm{\tilde{y}_t-\tilde{x}_t}},X\setminus C)\geq\kappa$ and hence,
\begin{equation}\label{QP-ratio-bound}\begin{aligned}&\frac{\norm{P_{z_t}(\tilde{y}_t-\tilde{x}_t)}}{\norm{\tilde{y}_t-\tilde{x}_t}}\geq\kappa,\\
&\frac{\norm{Q_{z_t}(\tilde{y}_t-\tilde{x}_t)}}{\norm{P_{z_t}(\tilde{y}_t-\tilde{x}_t)}}\leq\frac{1+\kappa}{\kappa}.\end{aligned}\end{equation}

Let $\nu_{\kappa}$ be the constant function
mentioned in Lemma \ref{expand-in-Ez} and $m_{\kappa}=\max\limits_{z\in\omega(x)}\{\nu_{\kappa}(z)\}$. By the smoothness of $\Phi_t$,
there is a $\delta^{''}\in (0,\frac{\delta}{2})$ such that \begin{equation}\label{linear-perturbation}\norm{D_{y_1}\Phi_t-D_{y_2}\Phi_t}
<\frac{1}{2}\end{equation} for any $t\in[0,m_{\kappa}]$ and $y_1,y_2\in \mathcal{B}_{\delta}(\omega(x))$ satisfying $\norm{y_1-y_2}<2\delta^{''}$.

Now, we will prove that $\delta^{''}$ is the desired constant. Prove by contrary. Suppose that
$$\limsup\limits_{t\rightarrow+\infty}\norm{\tilde{y}_t-\tilde{x}_t}<\delta^{''}$$ for some $y\in X$ satisfying $y\neq x$ and $y\thicksim x$.
Then, one can find a $N_1>\tilde{T}_y$ such that
$$\norm{\tilde{y}_t-\tilde{x}_t}<\delta^{''}\,\,\text{and}\,\,\norm{\tilde{x}_t-z_t}<\delta^{''}$$ for any
$t\geq N_1$. Hence, $\tilde{y}_t,\tilde{x}_{t}\in \mathcal{B}_{\delta}(\omega(x))$ and (\ref{QP-ratio-bound}) hold for any $t>N_1+T$.
Take $\tau_1\geq N_1+T$ and $\tau_{n+1}=\tau_{n}+\nu_{\kappa}(z_{\tau_n})$ for $n=1,2,\cdots$. Denoted by $y_{\tau_n}=\Phi_{\tau_n}(y)$,
$x_{\tau_n}=\Phi_{\tau_n}(x)$ and $P_{\tau_n}=P_{z_{\tau_n}}$, $Q_{\tau_n}=Q_{z_{\tau_n}}$. Then, one has
$$\begin{aligned}\norm{D_{z_{\tau_n}}\Phi_{\nu_{\kappa}(z_{\tau_n})}(y_{\tau_n}-x_{\tau_n})}&\geq
\norm{D_{z_{\tau_n}}\Phi_{\nu_{\kappa}(z_{\tau_n})}P_{\tau_n}(y_{\tau_n}-x_{\tau_n})}\cdot\big [1-\frac{\norm{D_{z_{\tau_n}}
\Phi_{\nu_{\kappa}(z_{\tau_n})}Q_{\tau_n}(y_{\tau_n}-x_{\tau_n})}}{\norm{D_{z_{\tau_n}}\Phi_{\nu_{\kappa}(z_{\tau_n})}
P_{\tau_n}(y_{\tau_n}-x_{\tau_n})}}\big]\\
&\overset{(\ref{eq:nu-kappa})}{\geq}\norm{D_{z_{\tau_n}}\Phi_{\nu_{\kappa}(z_{\tau_n})}P_{\tau_n}(y_{\tau_n}-x_{\tau_n})}\cdot[1-\frac{\kappa}
{2(1+\kappa)}\cdot\frac{\norm{Q_{\tau_n}(y_{\tau_n}-x_{\tau_{n}})}}{\norm{P_{\tau_n}(y_{\tau_n}-x_{\tau_{n}})}}]\\
&\overset{(\ref{QP-ratio-bound})}{\geq}\norm{D_{z_{\tau_n}}\Phi_{\nu_{\kappa}(z_{\tau_n})}P_{\tau_n}(y_{\tau_n}-x_{\tau_n})}\cdot[1-\frac{\kappa}
{2(1+\kappa)}\cdot\frac{1+\kappa}{\kappa}]\\
&=\frac{1}{2}\norm{D_{z_{\tau_n}}\Phi_{\nu_{\kappa}(z_{\tau_n})}P_{\tau_n}(y_{\tau_n}-x_{\tau_n})}\\
&\overset{(\ref{eq:nu-kappa})}{\geq}\frac{2}{\kappa}\norm{P_{\tau_n}(y_{\tau_n}-x_{\tau_n})}\\
&\overset{(\ref{QP-ratio-bound})}{\geq}2\norm{y_{\tau_n}-x_{\tau_n}}\end{aligned}$$ for any $n>0$.
Notice that $$y_{\tau_{n+1}}-x_{\tau_{n+1}}=D_{z_{\tau_n}}\Phi_{\nu_{\kappa}(z_{\tau_n})}(y_{\tau_n}-x_{\tau_n})
+\int_{0}^{1}[D_{x_{\tau_n}+s(y_{\tau_n}-x_{\tau_n})}\Phi_{\nu_{\kappa}(z_{\tau_n})}
-D_{z_{\tau_n}}\Phi_{\nu_{\kappa}(z_{\tau_n})}]ds\cdot(y_{\tau_n}-x_{\tau_n})$$ for any $n>0$.
Together with (\ref{linear-perturbation}), one has $$\norm{y_{\tau_{n+1}}-x_{\tau_{n+1}}}\geq \frac{3}{2}\cdot
\norm{y_{\tau_{n}}-x_{\tau_{n}}}$$ for any $n>0$. Hence, $\lim\limits_{n\rightarrow\infty}\norm{y_{\tau_n}-x_{\tau_n})}=\infty$, a contradiction.

Therefore, we complete the proof.
\end{proof}

\begin{theorem}\label{omega-lambdak>0} Let $\mathcal{D}$ be an open $\omega$-compact set and $x\in \mathcal{D}$. If $\lambda_{kz}>0$ for any $z\in \omega(x)$, then we have $x\in \overline{Q}$.
\end{theorem}

\begin{proof} See \cite[Theorem 4.5]{F-W-W2}. We here point out that all arguments in \cite[Theorem 4.5]{F-W-W2} still remain valid for  semiflow $\Phi_t$ strongly focusing monotone w.r.t. $C$ on an infinite dimensional Banach space.
\end{proof}

We now consider the case that $\lambda_{k\tilde{z}}>0$ for any $\tilde{z}\in \omega_0(x)$ and $\lambda_{kz}\leq 0$ for some $z\in \omega(x)\setminus\omega_0(x)$.
\begin{theorem}\label{regular-lambdak>0} Let $\mathcal{D}$ be an open $\omega$-compact set and $x\in \mathcal{D}$. Assume that $\lambda_{k\tilde{z}}>0$ for any $\tilde{z}\in \omega_0(x)$. If there exists some $z\in\omega(x)\setminus \omega_0(x)$ such that $\lambda_{kz}\leq 0$, then $x\in \overline{Q}$.
\end{theorem}

\begin{proof} Since $\Phi_t$ admits a flow extension on $\omega(x)$, one can define a vector $v_y= \frac{d}{dt}\mid_{t=0}\Phi_t(y)\in X$ for any $y\in \omega(x)$. The smoothness of $\Phi_t$ implies that the map $y:\mapsto v_y$ is continuous. Hence, all arguments in \cite[Theorem 4.6]{F-W-W2} are still effective for semiflow $\Phi_t$ strongly focusing monotone with respect to $C$ on an infinite dimensional Banach space.
\end{proof}

Before discussing the case that $\lambda_{kz}\leq 0$ for a regular point $z\in\omega_0(x)$, we present the following lemma, which describes the nonlinear dynamics nearby a regular point.

\begin{lemma}\label{oribit-lambdak<0} Let $x$ be a regular point. If $\lambda_{kx}\leq 0$, then there exists an open neighborhood $\mathcal{V}$
of $x$ such that for any $y\in \mathcal{V}$, one of two following properties holds:
\vskip 1mm
$(a)$ $\norm{\Phi_{t}(x)-\Phi_{t}(y)}\rightarrow 0$ as $t\rightarrow +\infty$;

\vskip 2mm
$(b)$ There exists a $T>0$ such that $\Phi_{T}(x)-\Phi_{T}(y)\in  C$; an hence, $\Phi_{t}(x)-\Phi_{t}(y)\in \text{Int}\, C$ for any $t>T$.
\end{lemma}

\begin{proof} By repeating all arguments in \cite[Lemma 4.1]{F-W-W2}, the conclusion in this lemma can be obtained for $\Phi_t$ strongly focusing monotone w.r.t. $C$.
\end{proof}

\begin{theorem}\label{omega-lambdak<=0}
Assume that there exists a regular point $z\in\omega(x)$ satisfying $\lambda_{kz}\leq 0$. Then either $x\in Q$, or $\omega(x)=\{z\}$ is
a singleton.
\end{theorem}
\begin{proof}The result is implied by repeating all arguments in \cite[Theorem 4.2]{F-W-W2}.
\end{proof}

\section{Proofs of Theorem A and B}

\indent Due to the local behaviors in the previous section, we can now describe the {\it generic dynamics} of the semiflow $\Phi_t$ strongly focusing monotone w.r.t. $C$ (see Theorem A) on a general Banach space, which concludes that generic (open and dense) positive semiorbits are either pseudo-ordered or convergent to equilibria. When $k=2$, together with the results in Lemma \ref{P-B thm}, we will further show the Poincar\'{e}-Bendixson Theorem (see Theorem B), that is to say, for generic (open and dense) points, its $\omega$-limit set containing no equilibria is a periodic orbit.

\vskip 3mm
{\it Proof of Theorem A.}
\begin{proof}
We note that $\mathcal{D}$  is $\omega$-compact because $\mathcal{O}^+(\mathcal{D})$ is precompact. Then, for any $x\in \mathcal{D}$, one of the following three alternatives holds:

(a) $\lambda_{kz}>0$ for any $z\in \omega(x)$;

(b) $\lambda_{k\tilde{z}}>0$ for any regular point $\tilde{z}\in\omega_0(x)$ and $\lambda_{kz}\leq 0$ for some irregular point $z\in\omega(x)\setminus \omega_0(x)$;

(c) $\lambda_{kz}\leq 0$ for some regular point $z\in\omega(x)$.

\noindent By virtue of Theorem \ref{omega-lambdak>0}, \ref{regular-lambdak>0} and \ref{omega-lambdak<=0}, one has $x\in \overline{Q}\cup C_E$ for any $x\in \mathcal{D}$. Thus, $Q\cup C_{E}$ is dense in $\mathcal{D}$.

To prove that $\text{Int} (Q\cup C_{E})$ is dense in $\mathcal{D}$ by contrary, we suppose that there is an open subset $U$ in $\mathcal{D}$ such that $U\cap \text{Int} (Q\cup C_{E})=\emptyset$. By strongly monotonicity of $\Phi_t$, $Q$ is open. Together with Theorem \ref{omega-lambdak>0} and \ref{regular-lambdak>0}, case (a) and (b) will not occur for any point in $U$. Then, only case (c) can occur for any point in $U$. Moreover, by virtue of Theorem \ref{omega-lambdak<=0}, one has that $U\subset C_E$. Recall that $U$ is open. Thus, $U\subset \text{Int} (Q\cup C_{E})$, a contradiction.

Therefore, we complete the proof.
\end{proof}
\vskip 3mm
{\it Proof of Theorem B.}
\begin{proof}  By Theorem A, ${\rm Int}(Q\cup C_E)$ is open and dense in $\mathcal{D}$. Now, given any $x\in {\rm Int}(Q\cup C_E)\cap\mathcal{D}$, if $\omega(x)\cap E=\emptyset$, then $x\in Q$. It then follows from Lemma \ref{P-B thm}, $\omega(x)$ is a periodic orbit. Therefore, we complete the proof.
\end{proof}

\begin{rem} {\rm In this paper, we introduce the strongly focusing monotonicity. The works on the generic dynamics of semiflows is the beginning of the series of research on the strongly focusing monotonicity, which extends our previous works (see \cite{F-W-W2}) on the generic behaviors of flows strongly monotone with high-rank cones to a class of general semiflows on an infinite dimensional Banach space. In the sequel, we will report the prevalence behaviors of this class of semiflows, the results on the whole-order property and Poincar\'{e}-Bendixson property of omega limit sets and the works on extending the strongly focusing monotonicity from semiflows to discrete systems, cocycles and skew-product semiflows.}
\end{rem}

\indent{\bf Acknowledgments}

I thank Prof. Yi Wang and Prof. Jianhong Wu for many useful suggestions.

\end{document}